\documentclass{amsart}

\usepackage{amsfonts}
\usepackage{amsmath}
\usepackage{graphicx}

\graphicspath{{./FatButterfly.png/}}
\newtheorem{theorem}{Theorem}

\newtheorem{corollary}{Corollary}

\newtheorem{definition}{Definition}
\newtheorem{example}{Example}

\newtheorem{lemma}{Lemma}

\newtheorem{proposition}{Proposition}
\newtheorem{remark}{Remark}

\numberwithin{equation}{section}

\begin{document}
\title{A Normal Graph Algebra}
\author{Harold N. Ward}
\address{Department of Mathematics\\
University of Virginia\\
Charlottesville, VA 22904\\
USA}
\email{hnw@virginia.edu}

\begin{abstract}
We define a normal graph algebra modeled on algebras used in genetics.
Although the algebra does not always determine its graph, it often
highlights special features. After developing basic properties of the
algebra, we examine those of certain minimal graphs. We then apply the
results to the Petersen graph, finding connections between some of its many
aspects. For example, the outer automorphisms of $\mathrm{Sym}(6)$ emerge
naturally. The normal algebra of the Petersen graph is unique among normal
graph algebras.
\end{abstract}

\maketitle

\section{Introduction\label{SectIntro}}

Algebraic method s play a prominent role in the investigation of graph
properties. Many of these stem from the creation of an algebra from the
graph, such as the algebra of polynomials in the adjacency matrix. In \cite%
{CG} and \cite{GC}, the authors present a commutative nonassociative algebra
defined by the incidence properties of a graph. It is inspired by Bernstein
algebras introduced in genetic studies (for a survey of that topic, see \cite%
{R}). This Bernstein graph algebra determines the graph itself \cite{GC, W}.
In the present paper, we define another nonassociative algebra from a graph
that is a good bit simpler. It does not always determine the graph; for
example, all trees of a given order have isomorphic algebras. But it seems
to encapsulate certain graph properties. It leads to aspects of the graph in
a natural way, sometimes by means of the automorphism group of the algebra.

\subsection{Normal algebras}

Following the lead in \cite{GLM}, we define a \emph{normal algebra}:

\begin{definition}
\label{DefNormal}\emph{A normal algebra consists of a finite dimensional
vector space }$\mathcal{N}$\emph{\ over a field }$\mathbb{F}$\emph{\ of
characteristic not 2, endowed with a product }$(a,b)\longrightarrow ab$. 
\emph{(We may write }$a\times b$\emph{\ if juxtaposition is confusing.) The
product has the following properties:}
\end{definition}

\begin{enumerate}
\item Bilinearity: for $a,b,a^{\prime },b^{\prime }$ in $\mathcal{N}$ and $%
\alpha ,\beta ,\alpha ^{\prime },\beta ^{\prime }$ in $\mathbb{F}$,%
\[
(\alpha a+\beta b)(\alpha ^{\prime }a^{\prime }+\beta ^{\prime }b^{\prime
})=\alpha \alpha ^{\prime }(aa^{\prime })+\alpha \beta ^{\prime }(ab^{\prime
})+\beta \alpha ^{\prime }(ba^{\prime })+\beta \beta ^{\prime }(bb^{\prime
}). 
\]

\item Commutativity: $ab=ba$.

\item Grading: $\mathcal{N}$ is the direct sum of two subspace $U$ and $%
\mathfrak{Z}$ for which%
\begin{eqnarray*}
u\mathfrak{z} &=&0\text{ for }u\in U,\,\mathfrak{z}\in \mathfrak{Z}; \\
\mathfrak{zz}^{\prime } &=&0\text{ for }\mathfrak{z},\,\mathfrak{z}^{\prime
}\in \mathfrak{Z}; \\
uu^{\prime } &\in &\mathfrak{Z}\text{ for }u,u^{\prime }\in U.
\end{eqnarray*}
\end{enumerate}

\noindent This is the framework presented in \cite{GLM}. Bernstein algebras
are not necessarily associative: $a(bc)$ might not equal $(ab)c$. (Such an
algebra is traditionally called a nonassociative algebra, even though it
might accidentally \emph{be} associative.) However, normal algebras are
trivially associative, since any three-fold product is automatically $0$.
They are also special Jordan algebras, the algebra itself providing the
associative algebra for which $ab=\frac{1}{2}(ab+ba)$. (Chapter IV of the
standard reference on nonassociative algebra \cite{SNA} presents Jordan
algebras.) None of these extras is used in this paper.

Needed technical aspects of normal algebras will be given on the spot. For
example, a homomorphism of one algebra $\mathcal{A}$ to another $\mathcal{A}%
^{\prime }$ is a linear transformation $\varphi :\mathcal{A}\longrightarrow 
\mathcal{A}^{\prime }$ for which $\varphi (ab)=\varphi (a)\varphi (b)$. If $%
\mathcal{A}$ and $\mathcal{A}^{\prime }$ are normal, $\varphi $ is also
required to map the $U$- and $\mathfrak{Z}$-subspaces of $\mathcal{A}$ into
the respective subspaces of $\mathcal{A}^{\prime }$.

If $X$ is a subset of an $\mathbb{F}$-vector space, the \emph{span} $%
\left\langle X\right\rangle $ of $X$ is the set of $\mathbb{F}$-linear
combinations of the members of $X$, shortened to $\left\langle
x\right\rangle $ when $X=\left\{ x\right\} $. Two members $u$ and $v$ of a
vector space are called \emph{proportional }if $\left\langle u\right\rangle
=\left\langle v\right\rangle $, and we often indicate this by $u\approx
v$. Finally, for a vector space $V$, $\widehat{V}$ is its dual space of $%
\mathbb{F}$-linear functionals on $V$.

\section{Normal graph algebras\label{SectNGA}}

Let $G$ be a finite simple graph, one with no loops or multiple edges. The
set of vertices of $G$ is $VG$ and $EG$ is the set of edges. The \emph{order}
of $G$ is $p=\left\vert VG\right\vert $ and the \emph{size} is $q=\left\vert
EG\right\vert $; these notations will hold throughout. (For both basic
graph-theoretic concepts and topics in algebraic graph theory, see \cite{GR}%
.) Edges are two-element subsets of $VG$. Instead of the conventional
notations $\{x,y\}$ or $xy$ for the edge with vertices $x$ and $y$, we write 
$\left[ x,y\right] $, especially in the algebra frame-work where we see $xy=%
\left[ x,y\right] $. The vertices of an edge $[x,y]$ are called \emph{%
adjacent}, denoted $x\sim y$. Different edges are also called adjacent
when they share a vertex.

The \emph{normal graph algebra} $\mathcal{N}G$ of $G$ has for its $U$%
-subspace, $U_{G}$, the space of formal $\mathbb{F}$-linear combinations of
the vertices, and for its $\mathfrak{Z}$-subspace, $\mathfrak{Z}_{G}$, the
formal $\mathbb{F}$-linear combinations of the edges. Thus $%
U_{G}=\left\langle VG\right\rangle $ and $\mathfrak{Z}_{G}=\left\langle
EG\right\rangle $. The product in $\mathcal{N}G$ is defined for the basis
elements, the vertices and edges, and extended through bilinearity. (The
construction is similar to that in \cite[Section 2]{CG}. It also resembles
the definition of a group ring; the exposition in \cite[p. 172]{CvL} may be
helpful.) Specifically, if $x$ and $y$ are vertices,%
\begin{equation}
\begin{tabular}{l}
$xy\,(=x\times y)=\left\{ 
\begin{array}{c}
\left[ x,y\right] \text{ if }x\neq y\text{ and }x\sim y, \\ 
0\text{ if }x\neq y\text{ and }x\not \sim y.%
\end{array}%
\right. $ \\ 
$x^{2}=\sum \left[ x,y\right] $, summed over the $y$ adjacent to $x.$%
\end{tabular}
\label{ProdDef}
\end{equation}%
All other basis products, such as $x\times \mathfrak{e}$, $x\in VG,\,%
\mathfrak{e}\in EG$, are $0$, in line with the definition \ref{DefNormal}.
Let two members of $U_{G}$ be $u=\sum_{x\in VG}\theta _{x}x\,$and $%
v=\sum_{x\in VG}\eta _{x}x$. Then%
\[
uv=\left( \sum_{x\in VG}\theta _{x}x\right) \left( \sum_{x\in VG}\eta
_{x}x\right) =\sum_{x,y\in VG}\theta _{x}\eta _{y}xy. 
\]%
An edge $[a,b]$ appears in the last sum four times: in $a^{2},\,ab,\,ba$,
and $b^{2}$. The contributions from these terms are 
\[
\begin{tabular}{cccc}
$a^{2}$ & $ab$ & $ba$ & $b^{2}$ \\ 
$\theta _{a}\eta _{a}$ & $\theta _{a}\eta _{b}$ & $\eta _{b}\theta _{a}$ & $%
\theta _{b}\eta _{b}$%
\end{tabular}%
. 
\]%
Their sum is $\theta _{a}\eta _{a}+\theta _{a}\eta _{b}+\eta _{b}\theta
_{a}+\theta _{b}\eta _{b}=(\theta _{a}+\theta _{b})(\eta _{a}+\eta _{b})$.
Thus%
\begin{equation}
uv=\sum_{\left[ x,y\right] \in EG}(\theta _{x}+\theta _{y})(\eta _{x}+\eta
_{y})[x,y]\text{ and }u^{2}=\sum_{\left[ x,y\right] \in EG}(\theta
_{x}+\theta _{y})^{2}[x,y].  \label{ProdCoef}
\end{equation}

\begin{example}
\label{ExampleTrees}\emph{Here are two trees: let }$T_{1}$\emph{\ be the }%
path\emph{\ of order }$4$\emph{\ with vertices }$a,b,c,d$\emph{\ and edges}%
\[
\mathfrak{e}=[a,b],\,\mathfrak{f}=[b,c],\,\mathfrak{g}=[c,d]. 
\]%
\emph{Let }$T_{2}$\emph{\ be the }claw\emph{\ with the same vertices, but
edges}%
\[
\mathfrak{e}=[a,b],\,\mathfrak{f}=[a,c],\,\mathfrak{g}=[a,d]. 
\]%
\emph{First define elements }$u_{0},\,u_{1},\,u_{2},\,u_{3}$\emph{\ in }$%
T_{1}$\emph{\ and }$T_{2}$\emph{\ by}%
\[
\begin{tabular}{ccccc}
& $u_{0}=$ & $u_{1}=$ & $u_{2}=$ & $u_{3}=$ \\ 
$T_{1}$ & $a-b+c-d$ & $a$ & $\frac{1}{2}(a-b-c+d)$ & $d$ \\ 
$T_{2}$ & $a-b-c-d$ & $b$ & $c$ & $d$%
\end{tabular}%
. 
\]%
\emph{The }$u_{i}$\emph{\ form bases of the two }$U$\emph{-spaces. As a
computational example, expanding out }$u_{0}u_{2}\ $\emph{in }$T_{1}$ \emph{%
directly from the definition (\ref{ProdDef})\ gives}%
\begin{eqnarray*}
u_{0}u_{2} &=&\frac{1}{2}(a-b+c-d)(a-b-c+d) \\
&=&\frac{1}{2}(a^{2}-ab-ac+ad-ba+b^{2}+bc-bd \\
&&+ca-cb-c^{2}+cd-da+db+dc-d^{2}) \\
&=&\frac{1}{2}(\mathfrak{e-e-}0+0-\mathfrak{e+e+f+f-}0 \\
&&+0-\mathfrak{f-f-g+g-}0+0+\mathfrak{g-g}).
\end{eqnarray*}%
\emph{The whole-scale cancelling shows that }$u_{0}u_{2}=0.$ \emph{The
computation of }$u_{0}u_{2}$ \emph{is much easier when done by the product
formulas (\ref{ProdCoef}), because for each edge }$\left[ a,b\right] ,\,%
\left[ b,c\right] ,\,\left[ c,d\right] $ of $\mathcal{T}_{1}$\emph{, the
sums }$\theta _{x}+\theta _{y}$\emph{\ for }$u_{0}$ \emph{are all }$0$\emph{.%
}

\emph{The algebras }$\mathcal{N}T_{1}$\emph{\ and }$\mathcal{N}T_{2}$\emph{\
have the same product table in terms of the }$u_{i}$\emph{\ and the edges:}%
\[
\begin{tabular}{lllll}
& $u_{0}$ & $u_{1}$ & $u_{2}$ & $u_{3}$ \\ 
$u_{0}$ & $0$ & $0$ & $0$ & $0$ \\ 
$u_{1}$ & $0$ & $\mathfrak{e}$ & $0$ & $0$ \\ 
$u_{2}$ & $0$ & $0$ & $\mathfrak{f}$ & $0$ \\ 
$u_{3}$ & $0$ & $0$ & $0$ & $\mathfrak{g}$%
\end{tabular}%
. 
\]%
\emph{Hence $\mathcal{N}T_{1}\emph{\ }$and$\mathcal{N}T_{2}\emph{\ }$are
isomorphic, an illustration of the introductory comment about trees.}
\end{example}

The normal graph algebra of a single vertex with no edges is the
one-dimensional $\mathbb{F}$-algebra $\mathbb{F}^{0}$ with all products $0$. 
$\mathbb{F}^{0}$ is a direct summand of the normal graph algebra of any
bipartite graph. In general, $\mathcal{N}(G\cup H)$ is isomorphic to the
direct sum of $\mathcal{N}G$ and $\mathcal{N}H$. For a connected bipartite
graph of order at least 2, that $\mathbb{F}^{0}$ of a bipartite graph is not
the algebra of a subgraph.

\section{Short homomorphisms}

The \emph{short algebra} is the graph algebra of the \emph{spline} graph, $S$%
, that has one vertex, $r$, and one edge, $\mathfrak{s}$, with $r$ the sole
endpoint of $\mathfrak{s}$. This spline algebra $\mathcal{S}$ has the single
defining relation $r^{2}=\mathfrak{s}$. For $\mathcal{S}$, $%
U_{S}=\left\langle r\right\rangle $, the one-dimensional subspace spanned by 
$r$, and $\mathfrak{Z}_{S}=\left\langle \mathfrak{s}\right\rangle $. A \emph{%
short homomorphism} $\varphi $ from $\mathrm{alg}G$ (or any normal algebra)
to $\mathcal{S}$ is a normal homomorphism, meaning that in addition to the
standard algebra homomorphism rule that $\varphi $ is a linear
transformation for which $\varphi (mn)=\varphi (m)\varphi (n)$ when $m,n\in 
\mathrm{alg}G$, we also have that $\varphi $ sends $U_{G}$ into $U_{S}$ and $%
\mathfrak{Z}_{G}$ into $\mathfrak{Z}_{S}$. This implies that $\varphi $ is
described by two linear functionals, $\lambda $ and $\mu $, for which $%
\varphi (u)=\lambda (u)r$ and $\varphi (\mathfrak{z})=\mu (\mathfrak{z})s$,
where $u\in U_{G}$ and $\mathfrak{z}\in \mathfrak{Z}_{G}$. The homomorphism
rule for $\varphi $ that $\varphi (u)\varphi (v)=\varphi (uv)$ for all $%
u,v\in U_{G}$ is equivalent to the demand that $\mu (uv)=\lambda (u)\lambda
(v)$. Given a short homomorphism $\varphi $, we get $\lambda $ and $\mu $
satisfying the last equation. If we were just presented with a $\lambda $
and a $\mu $ and define $\varphi $ by $\varphi (u)=\lambda (u)r$ and $%
\varphi (\mathfrak{z})=\mu (\mathfrak{z})s$, then to verify that $\varphi $
is a short homomorphism, we would need to show that $\mu (uv)=\lambda
(u)\lambda (v)$ for all $u,v\in U_{G}$.

However, by the linearity of the functionals, this last requirement for
producing a short homomorphism from $\lambda $ and $\mu $ just has to be
checked when $u$ and $v$ are taken to be vertices $x$ and $y$. If $x\neq y$,
then $xy=0$ if $x$ and $y$ are not adjacent. So in that case, $\mu (xy)=0$,
and then $\lambda (x)\lambda (y)=0$. That is, one or both of $\lambda
(x),\lambda (y)$ must be $0$. If $x$ and $y$ are adjacent, then $xy=[x,y]$
and $\mu ([x,y])=\lambda (x)\lambda (y)$. Furthermore, since $x^{2}=\sum
[x,t]$ summed over the vertices $t$ adjacent to $x$, $\mu (x^{2})=\sum \mu
([x,t])$. Substituting from the previous equations gives $\lambda
(x)^{2}=\sum \lambda (x)\lambda (t)$. If $\lambda (x)=0$, this is
automatically true. But if $\lambda (x)\neq 0$, we can divide through and
get $\lambda (x)=\sum \lambda (t)$. Thus the requirements on $\lambda $ and $%
\mu $ are:

\begin{equation}
\begin{tabular}{l}
1. If $x$ and $y$ are nonadjacent vertices, at least one of $\lambda (x)$
and $\lambda (y)$ is $0$. \\ 
2. If $\lambda (x)\neq 0$, then $\lambda (x)=\sum \lambda (t)$, summed over
the vertices $t$ adjacent to $x$. \\ 
3. If $x$ and $y$ are adjacent vertices, then $\mu ([x,y])=\lambda
(x)\lambda (y)$.%
\end{tabular}
\label{ShortReq}
\end{equation}

\noindent Call such a $\lambda $ a \emph{short functional} (always assumed
nonzero). Given $\lambda \neq 0$ satisfying these conditions and then $\mu $
from the third one (extended linearly to all of $\mathfrak{Z}_{G}$), the map 
$\varphi $ defined by $\varphi (u+z)=\lambda (u)r+\mu (z)s$ will be a short
homomorphism on $\mathrm{alg}G$. Scaling such a $\lambda $ by $\alpha \neq 0$
will produce a short homomorphism also; it is not a scalar multiple of the $%
\varphi $ for $\lambda $ because $\mu $ is scaled by $\alpha ^{2}$. Call it
a scaled version of $\varphi $ nevertheless, and indicate the collection of
these \textquotedblleft scalings\textquotedblright\ of $\varphi $ by $%
\left\langle \varphi \right\rangle $. We also collect the scalar multiples
of $\lambda $ (including $0$) into the set $\left\langle \lambda
\right\rangle $. This is a one-dimensional subspace of $\widehat{U_{G}}$,
the dual space of $U_{G}$, and so a point of the projective space $P\widehat{%
U_{G}}$, the set of one-dimensional subspaces of $\widehat{U_{G}}$.

\section{Edges and short homomorphisms}

There is a useful relation between short homomorphisms and edges. Begin with
an edge $\mathfrak{e}=[a,b]$ of $G$. Let $\lambda _{\mathfrak{e}}\in 
\widehat{U_{G}}$ be defined as follows: $\lambda _{\mathfrak{e}}(a)=\lambda
_{\mathfrak{e}}(b)=1$ and $\lambda _{\mathfrak{e}}(c)=0$ for all vertices
other than $a$ and $b$. Then extend $\lambda $ linearly to all of $U_{G}$
(so $\lambda _{\mathfrak{e}}(\alpha a+\beta b+\ldots )=\alpha +\beta +\ldots 
$). Next, let $\mu _{\mathfrak{e}}([a,b])=1$ and let $\mu _{\mathfrak{e}%
}([x,y])=0$ for all other edges $[x,y]$ of $G$; again, extend this to all of 
$\mathfrak{Z}_{G}$ by linearity. It is easy to verify that $\lambda _{%
\mathfrak{e}}$ satisfies the requirements (\ref{ShortReq}) for producing a
short homomorphism. For example, the second requirement just comes down to $%
\lambda _{\mathfrak{e}}(a)=\lambda _{\mathfrak{e}}(b)$, the only case
needing examination.

What is noteworthy is that, up to scalars, these are the \emph{only} short
functionals. To see this, let $\lambda $ be a short functional and let $K$
be the set of vertices $x$ for which $\lambda (x)\neq 0$. If $x,y\in K$ and $%
x\neq y$, then as $\mu (xy)=\lambda (x)\lambda (y)$ and neither factor is $0$%
, $\mu (xy)\neq 0$. That means $[x,y]$ is an edge. Thus any two members of $K
$ are adjacent, so that $K$ is a \emph{clique} in $G$. By the summation
condition 2 in (\ref{ShortReq}), when $x\in K$, then $\lambda (x)=\sum
\lambda (t)$, summed over the $t$ adjacent to $x$. The only $t$ adjacent to $%
x$ that contribute are the members in $K$, since the ones outside $K$ have $%
\lambda (t)=0$. So $\lambda (x)=\sum_{t\in K-\{x\}}\lambda (t)$. Add $%
\lambda (x)$ to both sides to get $2\lambda (x)=\sum_{t\in K}\lambda (t)$.
The right side doesn't depend on $x$; call it $\kappa $: $\lambda (x)=\kappa
/2$ for all $x\in K$. Then all the $\lambda (t)$ in $\sum_{t\in K}\lambda (t)
$ are $k/2$, and we get $\kappa =|K|\times \kappa /2$. This has to be
interpreted as an equation in $\mathbb{F}$; simplified, it reads $|K|=2$.
But we can't conclude that $|K|$ is \emph{numerically} $2$ unless either $%
|K|<2+\mathrm{char}\mathbb{F}$ or $\mathrm{char}\mathbb{F}=0$. That will
certainly be the case when $\mathrm{char}\mathbb{F}\neq 0$ if the number of
vertices $p$ of $G$ has $p<2+\mathrm{char}\mathbb{F}$. If $G$ has no
complete subgraphs of order at least $2+\mathrm{char}\mathbb{F}$ when $%
\mathrm{char}\mathbb{F}\neq 0$, then we can still conclude that $|K|=2$. So
we'll restrict $\mathbb{F}$ either to have characteristic $0$ or to have $%
p<2+\mathrm{char}\mathbb{F}$.

That assumption being made, we conclude that $|K|$ is really $2$. This means 
$K$ is the pair of vertices $a,b$ of some edge, $[a,b]$. As we can scale $%
\lambda $, we may take $\lambda (a)=\lambda (b)=1$ and all other vertex
values $\lambda (c)=0$. Then the only edge having nonzero $\mu $-value is $%
[a,b]$. Thus we have exactly the recipe for the short homomorphism
corresponding to $[a,b]$. In summary:

\begin{theorem}
Let $\mathbb{F}$ be restricted by having either $\mathrm{char}\mathbb{F}=0$
or $p<2+\mathrm{char}\mathbb{F}$, $p$ the order of graph $G$. Then up to
scalars, the short functionals $\lambda $ are exactly those produced from
each edge $[a,b]$ by the assignment $\lambda (a)=\lambda (b)=1$, with all
other vertex values $\lambda (c)=0$.
\end{theorem}

With $u=\sum_{x\in VG}\theta _{x}x$, $\lambda _{\mathfrak{e}}(u)=\sum_{x\in
VG}\lambda _{\mathfrak{e}}(x)$. Then if $\mathfrak{e}=[x,y]$, $\lambda _{%
\mathfrak{e}}(u)=\theta _{x}+\theta _{y}$. Thus in the product equation (\ref%
{ProdCoef}),%
\begin{equation}
uv=\sum_{\mathfrak{e}}\lambda _{\mathfrak{e}}(u)\lambda _{\mathfrak{e}}(v)%
\text{ and }u^{2}=\sum_{\mathfrak{e}}\lambda _{\mathfrak{e}}(u)^{2}.
\label{EqnProdLambda}
\end{equation}

\section{Weights\label{SectWt}}

In this section we assign weights to the members of $\mathfrak{Z}$ by using
short homomorphisms. For future use, let $\Lambda $ be a collection of
nonzero scalar multiples of the $\lambda _{\mathfrak{e}}$, one for each edge 
$\mathfrak{e}$, and $\mathrm{M}$ a similar set for the $\mu _{\mathfrak{e}}$.

\begin{definition}
\label{DefWt}\emph{Let }$\mathfrak{z}\in \mathfrak{Z}_{G}$\emph{. The }%
support $\mathrm{supp}\left( \mathfrak{z}\right) $\emph{\ is the set of
edges appearing with a nonzero coefficient when }$\mathfrak{z}$\emph{\ is
written as a linear combination of edges. The }weight $\mathrm{wt}\left( 
\mathfrak{z}\right) $\emph{\ of }$\mathfrak{z}$\emph{\ is the size }$%
\left\vert \mathrm{supp}\left( \mathfrak{z}\right) \right\vert $.
\end{definition}

Since $\mu _{\mathfrak{e}}(\mathfrak{f})=0$ for any edge $\mathfrak{f}$
other than $\mathfrak{e}$, $\mathrm{wt}\left( \mathfrak{z}\right) $ is the
number of $\mu _{\mathfrak{e}}$ for which $\mu _{\mathfrak{e}}(\mathfrak{z}%
)\neq 0$. That count would be the same if we used functionals proportional
to the $\mu _{\mathfrak{e}}$. Hence

\begin{lemma}
\label{LemWt}For $\mathfrak{z}\in \mathfrak{Z}_{G}$, $\mathrm{wt}\left( 
\mathfrak{z}\right) $ is the number of members $\mu $ of $\mathrm{M}$ for
which $\mu (\mathfrak{z})\neq 0$.
\end{lemma}

\noindent The edges themselves are proportional to the $\mathfrak{z}\in 
\mathfrak{Z}_{G}$ for which $\mathrm{wt}\left( \mathfrak{z}\right) =1$.
Consequently we can determine the edges, to scalars, entirely from the
structure of $\mathcal{N}G$.

By equation (\ref{EqnProdLambda}), for $u,v\in U_{G}$,%
\[
\mathrm{supp}\left( u^{2}\right) =\left\{ \mathfrak{e}\in EG|\lambda _{%
\mathfrak{e}}(u)\neq 0\right\} \text{ and }\mathrm{supp}\left( v^{2}\right)
=\left\{ \mathfrak{e}\in EG|\lambda _{\mathfrak{e}}(v)\neq 0\right\} . 
\]%
Then it also follows from (\ref{EqnProdLambda}) that%
\begin{equation}
\mathrm{supp}\left( uv\right) =\mathrm{supp}\left( u^{2}\right) \mathbf{\cap 
}\mathrm{supp}\left( v^{2}\right) .  \label{EqnSprt}
\end{equation}

\subsection{The annihilator}

The definition of $\mathcal{N}G$ implies that $\mathfrak{Z}_{G}$ annihilates
the algebra. Modifying the annihilator concept a bit, we let $\mathrm{ann}G$%
, the \emph{annihilator} of $\mathcal{N}G$, mean just the set of $u\in U_{G}$
for which $uU_{G}=0$.

\begin{proposition}
\label{PropAnn}The annihilator $\mathrm{ann}G$ is the set of $u\in U_{G}$
for which $u^{2}=0$. It is also the intersection $\mathbf{\cap }_{\lambda
\in \Lambda }\ker \lambda $ of the kernels of the $\lambda \in \Lambda $.
\end{proposition}

\begin{proof}
If $u\in \mathrm{ann}G$, then certainly $u^{2}=0$. Conversely, if $u^{2}=0$,
then $\mathrm{supp}\left( u^{2}\right) =\emptyset $, making $\mathrm{supp}%
\left( uv\right) =\emptyset $ by (\ref{EqnSprt}). So $uv=0$ for any $v\in
U_{G}$. We have $u^{2}=0$ just when $\lambda _{\mathfrak{e}}(u)=0$ for each
edge $\mathfrak{e}$, by (\ref{EqnProdLambda}). As $\lambda _{\mathfrak{e}%
}(u)\approx \lambda (u)$ for that $\lambda \in \Lambda $ corresponding
to $\mathfrak{e}$, $u^{2}=0$ exactly when $u\in \mathbf{\cap }_{\lambda \in
\Lambda }\ker \lambda $.
\end{proof}

\begin{corollary}
\label{CorAnn}If $G$ is a connected bipartite graph, then 
\[
\mathrm{ann}G=\left\langle \sum_{x\in X}x-\sum_{y\in Y}y\right\rangle , 
\]%
where $X$ and $Y$ are the two parts of $G$.
\end{corollary}

\begin{proof}
Let $u=\sum_{x\in VG}\theta _{x}x$. The requirement $\lambda _{\mathfrak{e}%
}(u)=0$ for an edge $\mathfrak{e}=[x,y]$ reads $\theta _{x}+\theta _{y}=0$,
that is, $\theta _{y}=-\theta _{x}$. Then the connectedness of $G$ forces
the coefficients for $x\in X$ to have one value, say $\theta $, and those
for $y\in Y$ to have the opposite, $-\theta $.
\end{proof}

\subsection{The incidence matrix}

The \emph{incidence matrix }$I(G)$ is the matrix with rows indexed by $V(G)$
and columns by $E(G)$, with $1$ in row $x$ and column $\mathfrak{e}$ if $x$
is a vertex of $\mathfrak{e}$, and $0$ if not \cite[p. 165]{GR}. Matrix $%
I(G) $ can also be viewed as displaying in column $\mathfrak{e}$ the values
of $\lambda _{\mathfrak{e}}$ on the vertices. The matrix obtained by
replacing $\lambda _{\mathfrak{e}}$ by its representative in $\Lambda $ has
the same rank as $I(G),$since it is a scaling of the columns. A row
dependence of $I(G)$ corresponds to a member $u=\sum_{x}\theta _{x}x$ for
which $\lambda _{\mathfrak{e}}(u)=0$ for all edges $\mathfrak{e}$; that is,
a $u$ in $\mathbf{\cap }_{\lambda \in \Lambda }\ker \lambda $. Thus by \ref%
{PropAnn}, $\mathrm{rank}I(G)=p-\dim \mathrm{\mathrm{ann}}G$, $p$ the order
of $G$. As each bipartite component of $G$, including isolated vertices,
contributes $1$ to $\dim \mathrm{ann}G$, the number of such components is $%
\dim \mathrm{ann}G$. Hence:

\begin{proposition}
\label{PropIncRank}The rank of the incidence matrix $I(G)$ of a graph $G$ of
order $p$ is given by 
\[
\mathrm{rank}I(G)=p-k_{b}(G), 
\]%
where $k_{b}(G)$ is the number of its bipartite connected components \cite[%
Theorem 8.2.1]{GR}. Thus this rank is determined by the algebra $\mathcal{N}%
G $.
\end{proposition}

However, $\mathcal{N}G$ does not in general determine the entire number $%
k(G) $ of connected components of $G$, as will be illustrated at the end of
the next section. A variant $I^{\prime }(G)$ of the incidence matrix of $G$
is presented in \cite[p. 24]{BI}: orient the edges of $G$ by assigning a
direction to each one. Then if $\mathfrak{e}=\left[ a,b\right] $ with the
direction of $\mathfrak{e}$ being $a\longrightarrow b$, change the $1$ in
the $\left( a,\mathfrak{e}\right) $ position of $I(G)$ to $-1$ . Proposition
4.3 of \cite{BI} now shows that $\mathrm{rank}I^{\prime }(G)=p-k(G)$.

\section{Squares of weight 1\label{SectSq1}}

As we noted, a member $u$ of $U_{G}$ has $\mathrm{wt}\left( u^{2}\right) =1$
exactly when $u^{2}\approx \mathfrak{e}$ for some edge $\mathfrak{e}$.
The product formula (\ref{EqnProdLambda}) shows that $u^{2}\approx 
\mathfrak{e}$ requires $\lambda _{\mathfrak{e}}(u)\neq 0$ and $\lambda _{%
\mathfrak{f}}(u)=0$ for all edges $\mathfrak{f}$ other than $\mathfrak{e}$.
If such a $u$ exists, we can scale it to have $\lambda _{\mathfrak{e}}(u)=1$%
, making $u^{2}=\mathfrak{e}$. The focus of this section will be connected 
\emph{edge-square} graphs, those for which \emph{every} edge is a square.

\begin{proposition}
\label{PropSqEdge}Let $G$ be connected and let $\mathfrak{e}=\left[ a,b%
\right] $ be an edge. Then $u^{2}=\mathfrak{e}$ for some $u\in U_{G}$ if and
only if one of the following holds:

\begin{enumerate}
\item $\mathfrak{e}$ is in an odd cycle and $G-\mathfrak{e}$ is bipartite;

\item $\mathfrak{e}$ is a bridge and at least one component of $G-\mathfrak{e%
}$ is bipartite.
\end{enumerate}
\end{proposition}

\begin{proof}
Suppose that $u=\sum_{x\in VG}\theta _{x}x$ and $u^{2}=\mathfrak{e}$. Then $%
\lambda _{\mathfrak{e}}(u)=\theta _{a}+\theta _{b}\neq 0$; so $\theta
_{a}\neq 0$, say. Follow a path starting at $a$ not beginning with $%
\mathfrak{e}$. Then to have $\mathfrak{\theta }_{x}+\mathfrak{\theta }_{y}=0$
for all other edges $[x,y]$ along that path, the coefficients $\mathfrak{%
\theta }_{x}$ will have to alternate between $\mathfrak{\theta }_{a}$ and $-%
\mathfrak{\theta }_{a}$. If the path reaches $b$, it must be that $\mathfrak{%
\theta }_{b}=\mathfrak{\theta }_{a}$ to avoid $\mathfrak{\theta }_{a}+%
\mathfrak{\theta }_{b}=0$. Thus if $\mathfrak{e}$ is in a cycle, the cycle
must be odd, so that $G$ is not bipartite. But $G-\mathfrak{e}$ is
bipartite, with $\left\{ x|\theta _{x}=\theta _{a}\right\} $ and $\left\{
y|\theta _{y}=-\theta _{a}\right\} $ giving the needed parts. If $\mathfrak{e%
}$ is not in a cycle, then deleting $\mathfrak{e}$ separates $G$ into two
connected components. The one containing $a$ is bipartite, shown again by
the two parts described but just for that component. Hence one of the two
conditions holds.

Conversely, if condition 1 holds, assign $\theta _{a}$ and $\theta _{b}$
both to be $1/2$. Then define $\theta _{x}=1/2$ for $x$ in the part of $G-%
\mathfrak{e}$ containing $a$ and $b$, and $\theta _{y}=-1/2$ for $y$ in the
other part. As any edge $[x,y]$ other than $\mathfrak{e}$ connects vertices
in different parts, $\theta _{x}+\theta _{y}=0$ and $u^{2}=\mathfrak{e}$. If
condition 2 holds, suppose that the component containing $a$ is bipartite,
with parts $X$ and $Y$, and $a\in X$. Put $\theta _{x}=1$ for $x\in X$ and $%
\theta _{y}=-1$ for $y\in Y$. For any $z$ in the other component containing $%
b$, put $\mathfrak{\theta }_{z}=0$. Again $u^{2}=\mathfrak{e}$.
\end{proof}

For characterizing edge-square graphs, the relevant graphs are \emph{%
unicyclic} graphs--connected graphs having exactly one cycle \cite{AH}. They
in turn are characterized among connected graphs as those for which order
and size are equal. Such a graph comes from a tree by adding an edge to
complete a cycle.

\begin{proposition}
\label{PropEdgeSq}If $G$ is a tree or a unicyclic graph whose cycle is odd,
then $G$ is edge-square.
\end{proposition}

\begin{proof}
Both statements follow from Proposition \ref{PropSqEdge}.
\end{proof}

Let $G$ be a connected edge-square graph. With $EG=\left\{ \mathfrak{e}%
_{1},\ldots ,\mathfrak{e}_{q}\right\} $, let $u_{i}\in U_{G}$ give $%
u_{i}^{2}=\mathfrak{e}_{i}$, $1\leq i\leq q$. By (\ref{EqSprt}), $\mathrm{%
supp}(u_{i}u_{j})=\emptyset $ for $i\neq j$, forcing $u_{i}u_{j}=0$. This
implies that the $u_{i}$ are linearly independent, so that $q\leq p$: as $G$
is connected, $q=p-1$ or $p$. Define a normal algebra $\mathcal{O}_{p}$ over 
$\mathbb{F}$ ($\mathcal{O}$ for \textquotedblleft one\textquotedblright )
with basis $w_{1},\ldots ,w_{p},\,\mathfrak{z}_{1},\ldots ,\mathfrak{z}_{p}$%
, by the relations%
\begin{equation}
\mathcal{O}_{p}:\left\{ 
\begin{tabular}{l}
$w_{i}^{2}=\mathfrak{z}_{i},\,1\leq i\leq p;$ \\ 
$w_{i}w_{j}=0,\,1\leq i,j\leq p,\,i\neq j;$ \\ 
all other basis products $0$.%
\end{tabular}%
\right.   \label{EqnOp}
\end{equation}%
(The spline algebra is $\mathcal{O}_{1}$.) Its $U$- and $\mathfrak{Z}$%
-spaces are $\left\langle w_{1},\ldots ,w_{p}\right\rangle $ and $%
\left\langle \mathfrak{z}_{1},\ldots ,\mathfrak{z}_{p}\right\rangle $. It is
easy to check that the short homomorphisms of $\mathcal{O}_{p}$ are
proportional to those given by $w_{i}\longrightarrow r$, $\mathfrak{z}%
_{i}\longrightarrow \mathfrak{s}$, and $w_{j}\longrightarrow 0$, $\mathfrak{z%
}_{j}\longrightarrow 0$ for $i\neq j$, where $1\leq i,\,j\leq p$.

\begin{theorem}
\label{ThmEdgeSq}For a connected edge-square graph $G$, there are the two
parameter possibilities $q=p-1$ or $q=p$ with corresponding realizations:

\begin{enumerate}
\item $q=p-1$: then $\mathcal{N}G$ is isomorphic to $\mathbb{F}^{0}\oplus 
\mathcal{O}_{p-1}$. In this case, $\dim \mathrm{ann}G=1$ and $G$ is a tree.

\item $q=p$: then $\mathcal{N}G$ is isomorphic to $\mathcal{O}_{p}$. Here $%
\dim \mathrm{ann}G=0$, and $G$ is unicyclic with an odd cycle.
\end{enumerate}
\end{theorem}

\begin{proof}
If $q=p-1$, $G$ is a tree and bipartite. As earlier, if the parts are $X$
and $Y$, $u_{0}=\sum_{x\in X}x-\sum_{y\in Y}y\in \mathrm{ann}G$. The
isomorphism comes from the remarks leading to $\mathcal{O}_{p}$, with $%
u_{i}\longleftrightarrow w_{i}$ and $\mathfrak{e}_{i}\longleftrightarrow 
\mathfrak{z}_{i}$, $1\leq i\leq p-1$. That $\dim \mathrm{ann}G=1$ follows
from the fact that if $w\in \left\langle w_{1},\ldots ,w_{p}\right\rangle $
and $w^{2}=0$, then $w=0$. So $\mathrm{ann}G=\left\langle u_{0}\right\rangle 
$, giving the $\mathbb{F}^{0}$ term.

When $q=p$, the matching is the same, but with $1\leq i\leq p$. Since a
unicyclic graph $G$ whose cycle is even is bipartite and $\mathrm{ann}G\neq
0 $, the cycle in item 2 must be odd.
\end{proof}

The direct sum $\mathcal{O}_{p_{1}}\oplus \mathcal{O}_{p_{2}}\oplus \ldots
\oplus \mathcal{O}_{p_{n}}$ is isomorphic to $\mathcal{O}_{p}$, with $p=\sum
p_{i}$. Thus if $G$ is the union of $n$ graphs of the type in item 2 of the
theorem, $\mathcal{N}G$ depends only on $p$ and not the number of components.

\section{Squares of weight 2\label{SectSq2}}

Now suppose that $G$ is a graph for which every \emph{pair} of edges is the
support of a square of a member of $U_{G}$: a\emph{\ pair-square} graph.
Assume that $G$ has no isolated vertices. We can also assume that $G$ is
connected. For if not, any pair of edges with one edge in one component and
the other in another would separately have each edge supporting a square
from its component. That means all edges would support squares and the graph
would be edge-square.

If $\mathrm{supp}(u^{2})=\left\{ \mathfrak{e},\mathfrak{f}\right\} $, then
by scaling $u$, we can take $\lambda _{\mathfrak{e}}(u)=1$ and have $u^{2}=%
\mathfrak{e}+\alpha ^{2}\mathfrak{f}$ for some $\alpha \neq 0$. If $v^{2}=%
\mathfrak{e}+\beta ^{2}\mathfrak{g}$ for some third edge $\mathfrak{g}$,
with $\beta \neq 0$ and $\lambda _{\mathfrak{e}}(v)=1$, then $%
(u-v)^{2}=\alpha ^{2}\mathfrak{f+}\beta ^{2}\mathfrak{g}$. It follows that
if for some fixed edge $\mathfrak{e}$, all the edge pairs containing $%
\mathfrak{e}$ are square supports, then \emph{all }edge pairs are square
supports. However, if $\mathfrak{e}$ is a square, then all edges support
squares (since now $\alpha =0$), and the graph is again edge-square. So we
assume that $G$ is connected with no isolated vertices, and that no edge
alone supports a square.

Fix one edge $\mathfrak{e}_{0}$ and on numbering the remaining edges $%
\mathfrak{e}_{1},\ldots ,\mathfrak{e}_{q-1}$, let $u_{i}^{2}=\mathfrak{e}%
_{0}+\alpha _{i}^{2}\mathfrak{e}_{i},\,1\leq i\leq q-1$, with $\alpha
_{i}=\lambda _{\mathfrak{e}_{i}}(u_{i})$. Then $u_{i}u_{j}=\mathfrak{e}_{0}$%
, $1\leq i,j\leq q-1$, $i\neq j$. The $u_{i}$ are linearly independent: if $%
\sum \beta _{i}u_{i}=0$, the product with $u_{j}$ gives $\left( \sum \beta
_{i}\right) \mathfrak{e}_{0}+\beta _{j}\alpha _{j}^{2}\mathfrak{e}_{j}=0$,
making $\beta _{j}=0$ for all $j$. In particular, $p\geq q-1$. Now set up a
normal $\mathbb{F}$-algebra $\mathcal{T}_{p}$ ($\mathcal{T}$\ for
\textquotedblleft two\textquotedblright ) with dimension $2p+1$ and basis $%
w_{1},\ldots ,w_{p},\,\mathfrak{z}_{0},\mathfrak{z}_{1},\ldots ,\mathfrak{z}%
_{p}$. Again, the $U$-space is $\left\langle w_{1},\ldots
,w_{p}\right\rangle $ and the $\mathfrak{Z}$-space $\left\langle \mathfrak{z}%
_{0},\mathfrak{z}_{1},\ldots ,\mathfrak{z}_{p}\right\rangle $. The defining
relations are%
\begin{equation}
\mathcal{T}_{p}:\left\{ 
\begin{tabular}{l}
$w_{i}w_{j}=\mathfrak{z}_{0}+\delta _{ij}\mathfrak{z}_{i},\,1\leq i,j\leq p;$
\\ 
all other basis products $0$.%
\end{tabular}%
\right.  \label{EqTp}
\end{equation}%
(Here $\delta _{ij}$ is the Kronecker delta.) It will be useful to have a
description of the short functionals of $\mathcal{T}_{p}$:

\begin{lemma}
\label{LemLambdaTp}Let $\varphi $ be a short homomorphism of the algebra $%
\mathcal{T}_{p}$, $p\geq 3$, with associated functionals $\lambda $ and $\mu 
$. Then $\varphi $ is proportional to one of the following:%
\[
\begin{tabular}{ll}
$\varphi _{0}:$ & $\lambda _{0}(w_{i})=1,\,\mu _{0}(\mathfrak{z}_{0})=1$, $%
\mu _{0}(\mathfrak{z}_{i})=0$, $1\leq i\leq p;$ \\ 
$\varphi _{i},\,1\leq i\leq p:$ & $\lambda _{i}(w_{j})=\delta _{ij},\,\mu
_{i}\left( \mathfrak{z}_{0}\right) =0$, $\mu _{i}(\mathfrak{z}_{j})=\delta
_{ij}$, $1\leq i,\,j\leq p.$%
\end{tabular}%
\]
\end{lemma}

\begin{proof}
Let $\lambda $ and $\mu $ be the functionals of a short homomrphism $\varphi 
$ of $\mathcal{T}_{p}$. First suppose that $\mu (\mathfrak{z}_{0})\neq 0$.
Then as $\lambda (w_{i})\lambda (w_{j})=\mu (\mathfrak{z}_{0})$ for $i\neq j$%
, all $\lambda (w_{i})$ must all have the same nonzero value. Scaling it to
be $1$ makes $\mu (\mathfrak{z}_{0})=1$ and then all $\mu (\mathfrak{z}%
_{i})=0$, from $\lambda (w_{i})^{2}=\mu (\mathfrak{z}_{0})+\mu (\mathfrak{z}%
_{i})$. This is the recipe for $\varphi _{0}$.

Now assume that $\mu (\mathfrak{z}_{0})=0$. Then for some $i$, $\mu (%
\mathfrak{z}_{i})\neq 0$. Since $\mathfrak{z}_{i}=w_{i}^{2}-\mathfrak{z}_{0}$%
, $\mu (w_{i}^{2})\neq 0$, making $\lambda (w_{i})\neq 0$. Then $\lambda
(w_{i})\lambda (w_{j})=\mu (\mathfrak{z}_{0})$ for $i\neq j$ reads $\lambda
(w_{i})\lambda (w_{j})=0$ and gives $\lambda (w_{j})=0$. This in turn shows
that $\mu (\mathfrak{z}_{j})=0$. Scaling produces $\varphi _{i}$.

That the descriptions of the $\varphi _{i}$ do give short homomorphisms is
straight-forward.
\end{proof}

\begin{theorem}
\label{ThmPairSq}Let $G$ be a connected graph such that every pair of edges
is the support of a square of a member of $U_{G}$. Suppose also that $G$ is
not an edge-square graph, so that $p\geq 4$. Then $q=p$ or $q=p+1$, and the
following hold:

\begin{enumerate}
\item $q=p$: $G$ is an even cycle and $\mathcal{N}G$ is isomorphic to $%
\mathbb{F}^{0}\oplus \mathcal{T}_{p-1}$.

\item $q=p+1$: $\mathcal{N}G$ is isomorphic $\mathcal{T}_{p}$.
\end{enumerate}
\end{theorem}

\begin{proof}
We saw above that $p\geq q-1$, so that $p-1\leq q\leq p+1$, since $G$ is
connected. The case $q=p-1$ is excluded, as then $G$ is a tree and
edge-square. If $q=p$, $G$ is unicyclic. A terminal edge supports the square
of its degree 1 vertex, making $G$ edge-square again. So $G$ is a cycle. If $%
p$ is odd, Proposition \ref{PropSqEdge} implies that $G$ is still
edge-square. Thus $p$ must be even. In that case, $\dim \mathrm{ann}G=1$, by
Corollary \ref{CorAnn}. Then $U_{G}=\mathrm{ann}G\oplus \left\langle
u_{1},\ldots ,u_{p-1}\right\rangle $, the $u_{i}$ as preceding the
definition of $\mathcal{T}_{p}$ (\ref{EqTp}). The subspace $\left\langle
u_{1},\ldots ,u_{p-1}\right\rangle +\mathfrak{Z}_{G}$ is a subalgebra
isomorphic to $\mathcal{T}_{p}$ by the correspondence $u_{i}\longrightarrow
w_{i}$, $\mathfrak{e}_{0}\longrightarrow \mathfrak{z}_{0}$, and $\mathfrak{e}%
_{i}\longrightarrow \mathfrak{\alpha }_{i}^{-2}\mathfrak{z}_{i}$, $1\leq
i\leq p-1$. That shows $\mathcal{N}G$ to be isomorphic to $\mathbb{F}%
^{0}\oplus \mathcal{T}_{q}$.

For $q=p+1$, $\mathrm{ann}G=0$, since now $U_{G}=\left\langle u_{1},\ldots
,u_{p}\right\rangle $. The same correspondence shows that $\mathcal{N}G$ is
isomorphic to $\mathcal{T}_{p}$.
\end{proof}

To see that for an even cycle $C$, $\mathcal{N}C$ really is a pair-square
graph, index the vertices around $C$ as $x_{0},\ldots ,x_{q-1}$. Reading
indices modulo $q$, we first have $x_{i}^{2}=[x_{i-1},x_{i}]+\left[
x_{i},x_{i+1}\right] $. For $[x_{i-1},x_{i}]$ and $\left[ x_{j},x_{j+1}%
\right] $ disjoint,%
\[
(x_{i}-x_{i+1}+x_{i+2}-\ldots +(-1)^{j-i}x_{j})^{2}=[x_{i-1},x_{i}]+\left[
x_{j},x_{j+1}\right] . 
\]%
Thus all pairs of edges do support squares.

What about pair-square graphs $G$ with $q=p+1$?

\begin{theorem}
\label{ThmPaddle}Let $G$ be a pair-square connected graph for which $q=p+1$.
(Such a graph is not edge-square.) Then $G$ is a doubly-odd \emph{paddle
graph}; that is, it has two edge-disjoint odd cycles either sharing one
vertex or connected by a path. Moreover, any such graph is pair-square.
\end{theorem}

\begin{proof}
Graph $G$ has no terminal edges (they are squares), so all vertex degrees
are at least 2. Those degrees being $\delta _{1},\ldots ,\delta _{p}$, $%
\sum_{i=1}^{p}\delta _{i}=2q$ gives $\sum_{i=1}^{p}(\delta _{i}-2)=2$. Then
on renumbering, either $\delta _{1}$ and $\delta _{2}$ are both $3$ and $%
\delta _{i}=2$ for $i\geq 3$, or else $\delta _{1}=4$ and $\delta _{i}=2$, $%
i\geq 2$. In the first case, vertices $x_{1}$ and $x_{2}$ might be connected
by three paths. But then two of the path lengths would have the same parity.
Removing an edge from the third path leaves a bipartite graph, making that
edge a square by Proposition \ref{PropSqEdge}. From the comments above, that
would show $G$ to be edge-square, which is excluded. So $x_{1}$ and $x_{2}$
are on single cycles joined by a path between $x_{1}$ and $x_{2}$. Both
cycles are odd, by the same squared-edge argument. In the second case, if $%
x_{1}$ has degree $4$, it is on two cycles meeting just at $x_{1}$. Again,
both cycles must be odd.

That a doubly-odd paddle graph is pair-square is a matter of coefficient
assignment verification, along the lines of the argument for even cycles
above.
\end{proof}

The classical \textquotedblleft butterfly\textquotedblright\ graph is the
doubly-odd paddle graph of order 5. All doubly-odd paddle graphs of the same
order $p$ have isomorphic normal algebras $\mathcal{T}_{p}$. The number of
such graphs is the number of pairs $\left\{ m,n\right\} $ of odd integers $%
n,m\geq 3$ with $n+m\leq p+1$, $n=m$ allowed. The sequence of counts,
starting at $p=5$, is presented in \cite[Sequence A008642]{S}. (The number
of unrestricted paddle graphs is also there \cite[Sequence A033638]{S}.)

\section{Edge coherence\label{SectEdgeCoherence}}

A set of distinct edges $\mathfrak{e}_{1},\ldots ,\mathfrak{e}_{n}$ in a
graph $G$ is called \emph{coherent} if there is a matching set of scalars $%
\alpha _{1},\ldots ,\alpha _{n}$, not all $0$, such that at each vertex of $%
G $, the sum of the scalars for the edges incident with that vertex is $0$.
The coherence is \emph{proper} if none of the $\alpha _{i}$ is $0$. Although
this concept seems to require knowledge of the edge-vertex incidence
relation, it actually depends only on the normal graph algebra $\mathcal{N}G$%
:

\begin{lemma}
\label{LemCohere}Let $\mathfrak{e}_{1},\ldots ,\mathfrak{e}_{n}$ be a set of
distinct edges of the graph $G$. Let $\lambda _{1},\ldots ,\lambda _{n}$ be
corresponding short functionals. Then the $\mathfrak{e}_{i}$ are coherent if
and only if the $\lambda _{i}$ are linearly dependent.
\end{lemma}

\begin{proof}
Scaling the $\lambda _{i}$ does not affect their dependence, so we may
assume that $\lambda _{i}=\lambda _{\mathfrak{e}_{i}}$. Then $\sum \alpha
_{i}\lambda _{i}=0$ is a dependence just when at each vertex $x$, $\sum
\alpha _{i}\lambda _{i}(x)=0$. This in turn means that $\sum \alpha _{j}=0$,
summed over the edges $\mathfrak{e}_{j}$ incident with $x$. That is just the
condition for coherence, with $\alpha _{i}$ matching $\mathfrak{e}_{i}$.
\end{proof}

\begin{definition}
\label{DefMinCo}\emph{A graph }$G$\emph{\ is called} minimally coherent 
\emph{if its edges are coherent, but no proper subset of them is.}
\end{definition}

For instance, the assignment of $1$ and $-1$ alternately to the edges of an
even cycle shows it to be coherent. But no proper subset of edges is
coherent, since its edge-induced subgraph contains terminal edges that could
not be assigned nonzero scalars. The smallest minimally coherent graph is a
4-cycle.

As the dimension of $\widehat{U_{G}}$ is $p$, the order of $G$, any set of $%
p+1$ or more edges is coherent. So if $G$ is minimally coherent, its size $q$
is at most $p+1$.

\begin{theorem}
\label{ThmMinCo}The minimally coherent graphs are the even cycles and the
doubly-odd paddle graphs of Theorem \ref{ThmPaddle}.
\end{theorem}

\begin{proof}
The minimality of an even cycle was just noted. A doubly-odd paddle graph of
order $p$ has normal algebra isomorphic to $\mathcal{T}_{p}$ (Theorems \ref%
{ThmPairSq} and \ref{ThmPaddle}). The short functionals of $\mathcal{T}_{p}$
in Lemma \ref{LemPsiTp} satisfy $\lambda _{0}=\sum_{i=1}^{p}\lambda _{i}$,
but no proper subset of them is linearly dependent. Consequently, doubly-odd
paddle graphs are minimally coherent.

Conversely, let $G$ be a minimally coherent graph of order $p$ and size $q$.
Then $G$ is connected, since a coherence requires a coherence in at least
one component. So $q=p$ or $q=p+1$. At $q=p$, $G$ is unicyclic with no
terminal edges and so a cycle. Up to scaling, the only possible coherence in
a cycle is the alternating assignment of $1$ and $-1$ to its edges going
around it, and that does not work if the cycle is odd. Thus $G$ is an even
cycle.

Now let $q=p+1$. The proof runs as in Theorem \ref{ThmPaddle}. Once again, $%
G $ cannot contain any even cycles, because they would be graphs of lower
size with coherent edges. Thus $G$ is indeed a doubly-odd paddle graph.
\end{proof}

Incidentally, this theorem implies a variant of a well-known one of P\'{o}sa 
\cite{EP} on the existence of pairs of edge-disjoint cycles:

\begin{corollary}
\label{CorMinCo}Let $G$ be a graph of order $p$ and size $q$ with $p\geq 4$
and $q\geq p+1$. Then $G$ contains either an even cycle or a pair of
edge-disjoint odd cycles.
\end{corollary}

\begin{proof}
As pointed out, the edges of $G$ are coherent. A minimal coherent set of
edges induces a minimally coherent subgraph of $G$ providing the needed
cycle set. (A direct proof makes a good exercise!)
\end{proof}

\section{Automorphisms\label{SectAut}}

We plan to investigate automorphisms of normal graph algebras in a later
paper. Here we present a few comments. The automorphism group $\mathrm{aut}G$%
, the group of permutations of $VG$ whose induced actions on 2-element
subsets of $VG$ permute edges, induces automorphisms of $\mathcal{N}G$. We
call these \emph{graphical automorphisms} and also refer to their set as $%
\mathrm{aut}G$. Other automorphisms are \emph{nongraphical}. There are \emph{%
scalar automorphisms}, maps scaling the members of $U_{G}$ by a nonzero
scalar $\alpha $ and members of $\mathfrak{Z}_{G}$ by $\alpha ^{2}$. Their
subgroup of the automorphism group $\mathrm{aut}\mathcal{N}G$ of $\mathcal{N}%
G$ will be denoted $\mathbb{F}_{G}^{\#}$ ($\mathbb{F}^{\#}$ is the set of
nonzero members of $\mathbb{F}$).

For $g\in \mathrm{aut}\mathcal{N}G$ and a short homomorphism with
functionals $\lambda $ and $\mu $, define $g\varphi $ to be $\varphi \circ
g^{-1}$, $g\lambda =\lambda \circ g^{-1}$, and $g\mu =\mu \circ g^{-1}$.
Then $g\varphi $ is also a short homomorphism, with functionals $g\lambda $
and $g\mu $. Thus $g$ permutes the spans of the members of $\mathrm{M}$ and
so preserves weights: $\mathrm{wt}(g\mathfrak{z)=}\mathrm{wt}\left( 
\mathfrak{z}\right) $, $\mathfrak{z}\in \mathfrak{Z}_{G}$. The edges being
proportional to the $\mathfrak{z}\in \mathfrak{Z}_{G}$ of weight 1, $g$
permutes the edge spans $\left\langle \mathfrak{e}\right\rangle $. (For $%
g\in \mathrm{aut}G$, this is automatic.)

If $H$ is a spanning subgraph of $G$, one with $VH=VG$, $\mathcal{N}H$ is
not necessarily a subalgebra of $\mathcal{N}G$. But it is a quotient, by the
map $x\longrightarrow x$ for $x\in VG$, $\mathfrak{e}\longrightarrow 
\mathfrak{e}$ for $\mathfrak{e}\in EH$, and $\mathfrak{e}\longrightarrow 0$
for $\mathfrak{e}\notin EH$. The kernel is the rather trivial ideal $%
\left\langle EG-EH\right\rangle $, for which $w\left\langle
EG-EH\right\rangle =0$ for all $w\in \mathcal{N}G$. An automorphism $g$ of $%
\mathcal{N}G$ that permutes the edge spans of $EH$ (and so of $EG-EH$)
induces an automorphism $g|H$ of $\mathcal{N}H$.

An \emph{edge-scaling} automorphism is one that scales each edge. One might
hope that such an automorphism is scalar, but that may not be so. For
example, the automorphism group of $\mathcal{O}_{p}$ (\ref{EqnOp}) is
isomorphic to the monomial group on $\left\langle w_{1},\ldots
,w_{p}\right\rangle $. If the monomial matrix is $\left[ \alpha _{ij}\right] 
$, the matrix for the action on $\left\langle \mathfrak{z}_{1},\ldots ,%
\mathfrak{z}_{p}\right\rangle $ is $\left[ \alpha _{ij}^{2}\right] $. In
particular, diagonal transformations can scale the edges by arbitrary
nonzero squares.

Let $g$ be an edge-scaling automorphism. Then for each edge $\mathfrak{e}$, $%
g\lambda _{\mathfrak{e}}=\varepsilon _{\mathfrak{e}}\lambda _{\mathfrak{e}}$
for some $\varepsilon _{\mathfrak{e}}\in \mathbb{F}^{\#}$. Let $\mathfrak{e}%
_{1},\ldots ,\mathfrak{e}_{n}$ be a minimal coherent set of edges, with
corresponding short$\ $functionals $\lambda _{i}$, and let $\sum \alpha
_{i}\lambda _{i}=0$ be a dependence showing coherence. Applying $g$ gives $%
\sum \alpha _{i}\varepsilon _{\mathfrak{e}_{i}}\lambda _{i}=0$. By the
minimality, it must be that all $\varepsilon _{\mathfrak{e}_{i}}$ are the
same, otherwise some combination of the two dependency sums would have fewer
nonzero terms.

\begin{definition}
\label{DefMinCoConn}Let $\mathcal{MC}$ be the family of minimally coherent
graphs. A graph $G$ is called $\mathcal{MC}$\emph{-edge connected} if for
any two edges $\mathfrak{e}$ and $\mathfrak{f}$ of $G$, there is a sequence $%
G_{1},\ldots ,G_{m}$ of subgraphs $G_{i}$ of $G$ belonging to $\mathcal{MC}$
such that $\mathfrak{e}\in EG_{1}$, $\mathfrak{f}\in EG_{m}$, and $%
EG_{i}\cap EG_{i+1}\neq \emptyset $ for $1\leq i\leq m-1$.
\end{definition}

\noindent If such a graph has no isolated vertices, it is connected.

\begin{proposition}
\label{PropEdgeScal}Let $G$ be an $\mathcal{MC}$-edge connected graph for
which $\mathrm{ann}G=0$ ($G$ thus has no isolated vertices and is not
bipartite). If $g\in \mathrm{aut}\mathcal{N}G$ is edge-scaling, then $g$ is
a scalar automorphism.
\end{proposition}

\begin{proof}
If $H$ is a minimally coherent subgraph of $G$, then by the preceding
discussion, there is an $\varepsilon _{H}\in \mathbb{F}^{\#}$ for which $%
g\lambda _{\mathfrak{e}}=\varepsilon _{H}\lambda _{\mathfrak{e}}$ for all $%
\mathfrak{e}\in EH$. Then the edge overlap provided by Definition \ref%
{DefMinCoConn} implies the existence of an $\varepsilon \in \mathbb{F}^{\#}$
for which $g\lambda _{\mathfrak{e}}=\varepsilon \lambda _{\mathfrak{e}}$ for 
\emph{all} edges $\mathfrak{e}$ of $G$. Consequently $\lambda _{\mathfrak{e}%
}(g^{-1}u)=\lambda _{\mathfrak{e}}(\varepsilon u)$ for $u\in U_{G}$. Since $%
\mathbf{\cap }_{\mathfrak{e}\in EG}\ker \lambda _{\mathfrak{e}}=0$, by \ref%
{PropAnn}, $g^{-1}u=\varepsilon u$, that is, $gu=\varepsilon ^{-1}u$. Hence $%
g$ is a scalar automorphism, as wished.
\end{proof}

\section{The Petersen graph}

In this section we apply some of the above results to the Petersen graph, $P$
\cite{HS} (see \cite{BR} for data). Its vertices will be taken as the
2-element subsets of $\left\{ \mathbf{1,2,3,4,5}\right\} $, with $\left\{ 
\mathbf{i},\mathbf{j}\right\} $ abbreviated to $\mathbf{ij}$. Then $\mathbf{%
ij}\sim \mathbf{kl}$ just when the four indices are all different, and
we denote the corresponding edge by $\mathbf{ijkl}$. If $\mathbf{AB}$ is an
edge, where $\mathbf{A}$ and $\mathbf{B}$ are vertex pairs, then $\mathbf{BA}
$ is the same edge, and within $\mathbf{A}$ and $\mathbf{B}$ the two indices
can be switched. The symmetric group $\mathrm{Sym}(5)$ acts on $P$, and in
fact, $\mathrm{aut}P$ is isomorphic to it \cite[Theorem 4.6]{HS}.

The Peterson graph contains ten hexagons, all equivalent under $\mathrm{Sym}%
(5)$. A representative one has consecutive vertices $\mathbf{14},\mathbf{35},%
\mathbf{24},\mathbf{15},\mathbf{34},\mathbf{25}$. Its stabilizer is the
subgroup $\left\langle (123),\,\left( 12\right) ,\,\left( 45\right)
\right\rangle $. The three pairs of edges $\left\{ \mathbf{1435},\,\mathbf{%
2435}\right\} $, $\left\{ \mathbf{1435},\,\mathbf{1524}\right\} $, and $%
\left\{ \mathbf{1435},\,\mathbf{1534}\right\} $ represent the three orbits
of edge pairs: adjacent, skew, and opposite in a hexagon. Armed with this,
the proof of the following lemma is straight-forward (but perhaps tedious),
as are all the proofs from here on:

\begin{lemma}
\label{Lem15Wt3}Up to scalars, there are fifteen members $u$ of $U_{P}$ for
which\newline
$\mathrm{wt}\left( u^{2}\right) =3$: the ten vertices $\mathbf{ij}$ and the
five sums indicated in this table (the $-1$'s in $2u_{i}$ are at the
vertices showing an $\mathbf{i}$). The five $u_{i}$ form a $\mathrm{Sym}(5)$
orbit:%
\begin{equation}
\begin{tabular}{ccccccccccc}
& $\mathbf{12}$ & $\mathbf{13}$ & $\mathbf{14}$ & $\mathbf{15}$ & $\mathbf{23%
}$ & $\mathbf{24}$ & $\mathbf{25}$ & $\mathbf{34}$ & $\mathbf{35}$ & $%
\mathbf{45}$ \\ 
$2u_{1}$ & $-1$ & $-1$ & $-1$ & $-1$ & $1$ & $1$ & $1$ & $1$ & $1$ & $1$ \\ 
$2u_{2}$ & $-1$ & $1$ & $1$ & $1$ & $-1$ & $-1$ & $-1$ & $1$ & $1$ & $1$ \\ 
$2u_{3}$ & $1$ & $-1$ & $1$ & $1$ & $-1$ & $1$ & $1$ & $-1$ & $-1$ & $1$ \\ 
$2u_{4}$ & $1$ & $1$ & $-1$ & $1$ & $1$ & $-1$ & $1$ & $-1$ & $1$ & $-1$ \\ 
$2u_{5}$ & $1$ & $1$ & $1$ & $-1$ & $1$ & $1$ & $-1$ & $1$ & $-1$ & $-1$%
\end{tabular}
\label{EqnUi}
\end{equation}
\end{lemma}

Here are some representative products to which $\mathrm{Sym}(5)$ can be
applied to produce other products:%
\begin{equation}
\begin{tabular}{l}
$u_{1}^{2}=\mathbf{2345}+\mathbf{2435}+\mathbf{2534};$ \\ 
$u_{1}\times \mathbf{12}=0,\,u_{1}\times \mathbf{13}=0,\,u_{1}\times \mathbf{%
14}=0,\,u_{1}\times \mathbf{15}=0;$ \\ 
$u_{1}\times \mathbf{23}=\mathbf{2345},\,u_{1}\times \mathbf{45}=\mathbf{2345%
};$ \\ 
$u_{i}u_{j}=0,\,i\neq j.$%
\end{tabular}
\label{EqnUiProd}
\end{equation}%
Notice that $\mathbf{12}^{2}=\mathbf{1234}+\mathbf{1235}+\mathbf{1245}$, the
sum of the edges of the claw at $\mathbf{12}$, while $u_{1}^{2}$ is a sum of
three mutually opposite edges.

\begin{remark}
\label{RemEnNot}\emph{There is an enhanced notation that can be used here:
Label }$u_{i}$\emph{\ by }$\mathbf{i6},\,1\leq i\leq 5$\emph{, and augment
each edge symbol by the pair }$\mathbf{k6}$\emph{, where }$\mathbf{k}$\emph{%
\ does not appear in the original four edge indices. The earlier switching
rules extend to these symbols. Thus }$\mathbf{123456}=\mathbf{436512}$\emph{%
. Then the products involving the }$u_{i}$\emph{\ are obtained just as for }$%
\mathcal{N}P$\emph{\ itself. The products in (\ref{EqnUiProd}) now read:}%
\begin{equation}
\begin{tabular}{c}
$\mathbf{16}^{2}=\mathbf{162345}+\mathbf{162435}+\mathbf{162534};$ \\ 
$\mathbf{16}\times \mathbf{12}=0,\,\mathbf{16}\times \mathbf{13}=0,\,\mathbf{%
16}\times \mathbf{14}=0,\mathbf{16}\times \mathbf{15}=0;$ \\ 
$\mathbf{16}\times \mathbf{23}=\mathbf{162345},\,\mathbf{16}\times \mathbf{45%
}=\mathbf{162345};$ \\ 
$\mathbf{i6}\times \mathbf{j6}=0,\,i\neq j.$%
\end{tabular}
\label{Eqi6Prod}
\end{equation}%
\emph{We'll continue to use this notation.}
\end{remark}

There are six 5-element sets of the spans of these fifteen members of $U_{P}$
with the property that the products from any two different spans of the set
are $0.$They are%
\begin{equation}
\begin{tabular}{cc}
$\mathcal{F}_{1}$ & $\left\langle \mathbf{12}\right\rangle \ \left\langle 
\mathbf{13}\right\rangle \ \left\langle \mathbf{14}\right\rangle \
\left\langle \mathbf{15}\right\rangle \ \left\langle \mathbf{16}%
\right\rangle $ \\ 
$\mathcal{F}_{2}$ & $\left\langle \mathbf{12}\right\rangle \ \left\langle 
\mathbf{23}\right\rangle \ \left\langle \mathbf{24}\right\rangle \
\left\langle \mathbf{25}\right\rangle \ \left\langle \mathbf{26}%
\right\rangle $ \\ 
$\mathcal{F}_{3}$ & $\left\langle \mathbf{13}\right\rangle \ \left\langle 
\mathbf{23}\right\rangle \ \left\langle \mathbf{34}\right\rangle \
\left\langle \mathbf{35}\right\rangle \ \left\langle \mathbf{36}%
\right\rangle $ \\ 
$\mathcal{F}_{4}$ & $\left\langle \mathbf{14}\right\rangle \ \left\langle 
\mathbf{24}\right\rangle \ \left\langle \mathbf{34}\right\rangle \
\left\langle \mathbf{45}\right\rangle \ \left\langle \mathbf{46}%
\right\rangle $ \\ 
$\mathcal{F}_{5}$ & $\left\langle \mathbf{15}\right\rangle \ \left\langle 
\mathbf{25}\right\rangle \ \left\langle \mathbf{35}\right\rangle \
\left\langle \mathbf{45}\right\rangle \ \left\langle \mathbf{56}%
\right\rangle $ \\ 
$\mathcal{F}_{6}$ & $\left\langle \mathbf{16}\right\rangle \ \left\langle 
\mathbf{26}\right\rangle \ \left\langle \mathbf{36}\right\rangle \
\left\langle \mathbf{46}\right\rangle \ \left\langle \mathbf{56}%
\right\rangle $%
\end{tabular}
\label{EqnFi}
\end{equation}%
The group $\mathrm{aut}\mathcal{NP}$ permutes these six sets, with $\mathrm{%
Sym}(5)$ (as $\mathrm{aut}P$) permuting the first five by the subscripts. $%
\mathrm{Sym}(5)$ stabilizes $\mathcal{F}_{6}$, in which it permutes the $%
\left\langle \mathbf{k6}\right\rangle $ by the $k$'s.

\begin{proposition}
\label{PropFixFi}Suppose that $g\in \mathrm{aut}\mathcal{N}P$ fixes each $%
\mathcal{F}_{i}$. Then $g$ is a scalar automorphism.
\end{proposition}

\begin{proof}
From $\left\langle \mathbf{ij}\right\rangle =\mathcal{F}_{i}\cap \mathcal{F}%
_{j}$ we infer that $g$ fixes each vertex, up to scalars. Then $g$ is
edge-scaling. Since $P$ satisfies the hypotheses of Proposition \ref%
{PropEdgeScal}, $g$ is scalar. (In fact, this conclusion just needs $P$
connected.)
\end{proof}

As a result, we have a homomorphism of $\mathrm{aut}\mathcal{N}P$ into $%
\mathrm{Sym}(6)$, with kernel $\mathbb{F}_{P}^{\#}$.

\begin{theorem}
For the Petersen graph, $\mathrm{aut}\mathcal{N}P/\mathbb{F}_{P}^{\#}$ is
isomorphic to $\mathrm{Sym}(6)$ acting on $\left\{ \mathcal{F}_{1},\ldots ,%
\mathcal{F}_{6}\right\} $.
\end{theorem}

\begin{proof}
To prove this, we must find an additional automorphism outside of $\mathrm{%
Sym}(5)$ that produces a transposition. Define $t$ as follows; the members
of $U_{P}$ whose images are shown form a basis of $U_{P}$:%
\[
t:\mathbf{1k}\longrightarrow \mathbf{k6},\,2\leq k\leq 5;\,\mathbf{ij}%
\longrightarrow \mathbf{ij},\,2\leq i<j\leq 5.
\]%
Since $\sum_{k=2}^{5}\mathbf{1k}=\sum_{k=2}^{5}\mathbf{k6}$, $t$ fixes the
common sum and so fixes $\mathbf{16}$, which is $-\frac{1}{2}\sum_{k=2}^{5}%
\mathbf{1k}+\frac{1}{2}\sum_{2\leq i<j\leq 5}\mathbf{ij}$. Moreover, $t^{2}$
is the identity, so that $\mathbf{k6}\longrightarrow \mathbf{1k}$ for $2\leq
k\leq 5$. The induced map of the edges can be described this way: first, $t$
fixes any edge showing $\mathbf{16}$ as one of its three pairs. For an edge $%
\mathbf{1hijk6}$, $t\mathbf{1hijk6}=\mathbf{1kijh6}$; for instance, $t%
\mathbf{123456}=\mathbf{153426}$. It is routine to verify that $t$, with
these edge images, is an automorphism. (The fact that $t$ commutes with the
members of $\mathrm{Sym}(5)$ fixing $\mathbf{1}$ simplifies the work.) The
effect of $t$ on the $\mathcal{F}_{i}$ is that $t$ switches $\mathcal{F}_{1}$
and $\mathcal{F}_{6}$ and fixes the others. This $t$ is just the
transposition we need.
\end{proof}

In the notation of Remark \ref{RemEnNot}, $t$ is the transposition $(16)$.
Moreover, $\left\langle \mathrm{aut}P,t\right\rangle $ is isomorphic to $%
\mathrm{Sym}(6)$, and $\mathrm{aut}\mathcal{N}P$ is isomorphic to $\mathbb{F}%
^{\#}\times \mathrm{Sym}(6)$.

The group $\mathrm{Sym}(6)$ famously has outer automorphisms \cite{JR}, and
they are hiding in all this. There are six pairs of disjoint pentagons in $P$%
, and each pair occurs in five minimally coherent subgraphs of size 11. Two
pentagons share at most two edges, so that two such subgraphs involving
different pairs of pentagons meet in at most nine edges. So different
minimally coherent subgraphs with the same pentagon pair meet in that pair,
with its ten edges. By this count, the edge span sets of the pentagon pairs
can be identified from $\mathcal{N}P$, as can then the sets of five edge
spans showing the 11th edges of the subgraphs. The edges of such a set form
a 1-factor of $P$.

We obtain six sets $\mathcal{H}_{1},\ldots ,\mathcal{H}_{6}$ of five edge
spans. They are permuted by $\mathrm{aut}\mathcal{N}P$ (recorded on the
subscripts). The spans given are cycled by $\left( 12345\right) $ from $%
\mathrm{aut}P$, with $\mathcal{H}_{6}$ fixed.%
\begin{equation}
\begin{tabular}{llllll}
$\mathcal{H}_{1}$ & $\left\langle \mathbf{123456}\right\rangle $ & $%
\left\langle \mathbf{132546}\right\rangle $ & $\left\langle \mathbf{143526}%
\right\rangle $ & $\left\langle \mathbf{152436}\right\rangle $ & $%
\left\langle \mathbf{162345}\right\rangle $ \\ 
$\mathcal{H}_{2}$ & $\left\langle \mathbf{123546}\right\rangle $ & $%
\left\langle \mathbf{132456}\right\rangle $ & $\left\langle \mathbf{142536}%
\right\rangle $ & $\left\langle \mathbf{153426}\right\rangle $ & $%
\left\langle \mathbf{162345}\right\rangle $ \\ 
$\mathcal{H}_{3}$ & $\left\langle \mathbf{124536}\right\rangle $ & $%
\left\langle \mathbf{132546}\right\rangle $ & $\left\langle \mathbf{142356}%
\right\rangle $ & $\left\langle \mathbf{153426}\right\rangle $ & $%
\left\langle \mathbf{162435}\right\rangle $ \\ 
$\mathcal{H}_{4}$ & $\left\langle \mathbf{124536}\right\rangle $ & $%
\left\langle \mathbf{132456}\right\rangle $ & $\left\langle \mathbf{143526}%
\right\rangle $ & $\left\langle \mathbf{152346}\right\rangle $ & $%
\left\langle \mathbf{162534}\right\rangle $ \\ 
$\mathcal{H}_{5}$ & $\left\langle \mathbf{123456}\right\rangle $ & $%
\left\langle \mathbf{134526}\right\rangle $ & $\left\langle \mathbf{142536}%
\right\rangle $ & $\left\langle \mathbf{152346}\right\rangle $ & $%
\left\langle \mathbf{162435}\right\rangle $ \\ 
$\mathcal{H}_{6}$ & $\left\langle \mathbf{123546}\right\rangle $ & $%
\left\langle \mathbf{134526}\right\rangle $ & $\left\langle \mathbf{142356}%
\right\rangle $ & $\left\langle \mathbf{152436}\right\rangle $ & $%
\left\langle \mathbf{162534}\right\rangle $%
\end{tabular}
\label{EqnHi}
\end{equation}%
Once again, we have an isomorphism of $\mathrm{aut}\mathcal{N}P/\mathbb{F}%
_{P}^{\#}$ with $\mathrm{Sym}(6)$, but now acting on $\left\{ \mathcal{H}%
_{1},\ldots ,\mathcal{H}_{6}\right\} $.

From the actions on the $\mathcal{F}_{i}$ and the $\mathcal{H}_{i}$,
composition by way of $\mathrm{aut}\mathcal{N}P/\mathbb{F}_{P}^{\#}$
produces an automorphism of $\mathrm{Sym}(6)$. For instance, the
transposition $\left( 12\right) $ on the $\mathcal{F}_{i}$ effects $%
(15)(26)(34)$ on the $\mathcal{H}_{i}$, and $(123)$ produces $(164)(253)$.
The automorphism is indeed outer. The enhanced notation shows that a
duad-syntheme correspondence, suggested by J. J. Sylvester and described in 
\cite{C}, is present in (\ref{EqnHi}). Each edge appears in two rows, and
matching the row index pair with the edge sets up such a correspondence. For
example, $\mathbf{12}\longleftrightarrow \mathbf{162345}$ and $\mathbf{36}%
\longleftrightarrow \mathbf{142356}$. An outer automorphism also provides a
correspondence, matching transpositions with triple transpositions. For more
on this topic, see \cite[Chapter 6.]{CvL}.

We end with a uniqueness theorem for the algebra $\mathcal{N}P$, the first
such result in this paper. The sequel will contain other uniqueness theorems.

\begin{theorem}
\label{ThmNPUnique}Let $Q$ be a second graph of order $10$ and size $15$ for
which $\mathcal{N}P$ and $\mathcal{N}Q$ are isomorphic normal algebras. Then
the graphs $P$ and $Q$ are isomorphic.
\end{theorem}

\begin{proof}
Let $\varphi :\mathcal{N}Q\longrightarrow \mathcal{N}P$ be an isomorphism.
As in Section \ref{SectAut}, the $\varphi $-images of edge spans and their
associated form and functional spans are that sort of item of $\mathcal{N}P$%
. First observe that $P$ has no 4-cycles: its smallest minimally coherent
subgraphs are hexagons. Then $Q$ also has no 4-cycles, since $\varphi $
would match them with 4-cycles in $P$. Each edge of $P$ is in a hexagon, so
that is true of $Q$. This implies that $Q$ is connected. Next, $\mathcal{N}P$
contains no $u$ with $\mathrm{wt}\left( u^{2}\right) =1$ or $2$. This is
again a coefficient-assignment problem (automorphisms of $\mathcal{N}P$ may
be exploited to allow checking just the edge $\mathbf{123456}$ and the pair $%
\mathbf{123456}$, $\mathbf{132456}$ for potential supports). It follows that
the minimum degree of $Q$ is at least $3$, so that in fact $Q$ is regular of
degree $3$ (cubic). In addition, the $\varphi $-images of the vertices must
be proportional to ten of the 15 members of the set $\left\{ \mathbf{ij}%
|1\leq i<j\leq 6\right\} $ from Lemma \ref{Lem15Wt3}.

Suppose that among the missing five there are two that are disjoint, such as 
$\mathbf{12}$ and $\mathbf{36}$. Their product $\mathbf{124536}$ is an edge
that must also appear as a product from the ten. But that would require two
of $\mathbf{12},\,\mathbf{36}$, and $\mathbf{45}$ to be among the ten also,
which can't be. Therefore the excluded five form one of the sets described
after Remark \ref{RemEnNot}. Composing $\varphi $ with a member of $\mathrm{%
aut}\mathcal{NP}$, if needed, allows us to assume that the five are $\mathbf{%
16},\mathbf{26},\mathbf{36},\mathbf{46}$, and $\mathbf{56}$. Then $\varphi $
maps the vertices of $Q$ to scalar multiples of the vertices of $P$,
inducing a map $\widetilde{\varphi }$ of vertices. For $x,y\in VQ$, $%
x\sim y$ just when $xy\neq 0$. That in turn requires $(\varphi
x)(\varphi y)\neq 0$, which is equivalent to $\widetilde{\varphi }x\sim 
\widetilde{\varphi }y$. So $\widetilde{\varphi }$ preserves adjacency and
establishes the hoped-for isomorphism between $Q$ and $P$.
\end{proof}

There are 19 cubic graphs of order 10, tabulated with pertinent data in \cite%
[pp. 13--14]{B} (they are drawn in \cite[p. 127]{RW} too, also with data).
Their graph algebras are mutually nonisomorphic. For suppose that $G$ and $H$
are two of these graphs with $\mathcal{N}G$ and $NH$ isomorphic. Then $G$
and $H$ have the same number of minimally coherent subgraphs of each size.
Those of size 4 are the 4-cycles. There are none of size 5, and of the two
possibilities of size 6, only the 6-cycle can appear, because the butterfly
graph has a vertex of degree 4. Thus $G$ and $H$ must have the same numbers
of 4-cycles and 6-cycles.

There are just two pairs of the graphs for which these numbers match. First,
\#14 and \#18 (C10 and C19 in \cite{RW}) have two 4-cycles and eight
6-cycles. But the number of edges appearing in all the 4-cycles of \#14 is
seven, and in \#18, eight. Similarly, \#12 and \#17 (C16 and C11 in \cite{RW}%
) have three 4-cycles and seven 6-cycles. Again the numbers of edges in the
4-cycles are different: nine for \#12 and 11 for \#17. So all the algebras
are indeed nonisomorphic.

In further investigations, 4-cycles will play a prominent role.

\end{document}